\documentclass[11pt, a4paper]{article}
\usepackage[latin1]{inputenc}
\usepackage[T1]{fontenc}
\usepackage[english]{babel}
\usepackage{extarrows}
\usepackage{amssymb, amsmath, amsthm}
\usepackage{stmaryrd}
\usepackage[colorlinks=true, allcolors=blue]{hyperref}
\usepackage{geometry}
\usepackage{fancyhdr}
\usepackage{relsize}
\usepackage{upgreek}
\usepackage{verbatim}
\usepackage{enumitem}
\usepackage{calrsfs}
\usepackage{todonotes}
\usepackage[all]{xy}
\usepackage{tikz}
\usetikzlibrary{matrix,arrows,calc}
\usepackage{package}
\DeclareMathOperator{\gen}{gen}
\DeclareMathOperator{\GL}{GL}
\DeclareMathOperator{\Ind}{Ind}
\newcommand{\gl}{\mathfrak{gl}}
\renewcommand{\t}{\mathfrak{t}}
\newtheorem{que}[thm]{Question}

\def\udim{{\underline{\dim}\, }}

\newcommand{\Sc}[2]{\langle #1,\,#2\rangle}

\makeatletter
\renewcommand*{\@fnsymbol}{\@arabic}
\makeatother
\title{The Value of the Kac Polynomial at One}
\author{H. Franzen%
\thanks{Mathematisches Institut der Universit\"at Bonn, Endenicher Allee 60, 53115 Bonn\newline \href{mailto:franzen@math.uni-bonn.de}{franzen@math.uni-bonn.de}}
\and T. Weist%
\thanks{Bergische Universit\"at Wuppertal, School of Mathematics and Natural Sciences, Gau{\ss}str.\ 20, 42119 Wuppertal\newline \href{mailto:mreineke@uni-wuppertal.de}{mreineke@uni-wuppertal.de}}%
}
\date{}
\begin{document}
	\maketitle
	\begin{abstract}
		\noindent We establish a formula for the value of the Kac polynomial at one in terms of Kac polynomials, evaluated at one, of the universal (abelian) covering quiver by applying torus localization methods to quiver varieties introduced by Hausel--Letellier--Rodriguez-Villegas.
	\end{abstract}
	
	\section{Introduction}

	Given a quiver $\Gamma$ and a dimension vector $\alpha$, Kac defines in \cite{Kac:83} the function $a_{\Gamma,\alpha}(q)$ that counts the number of isomorphism classes of absolutely indecomposable representations of $\Gamma$ of dimension vector $\alpha$ over the finite field with $q$ elements. He shows that, regarded as a function in $q$, this defines a polynomial with integer coefficients---it is called the Kac polynomial. In fact the coefficients of the Kac polynomial are non-negative as shown by Hausel--Letellier--Rodriguez-Villegas in \cite{HLV:13}---confirming a conjecture of Kac.

	The main objective of this paper is to study the value of the Kac polynomial at one. We give a description of the number $a_{\Gamma,\alpha}(1)$ in terms of the values $a_{\hat{\Gamma},\beta}(1)$ where $\hat{\Gamma}$ is the universal abelian covering quiver of $\Gamma$. More precisely, the main result of the paper states:

	\begin{thm} \label{t:main}
		The value of the Kac polynomial at one of $\Gamma$ attached to a dimension vector $\alpha$ is the sum
		$$
			a_{\Gamma,\alpha}(1) = \sum_{\beta} a_{\hat{\Gamma},\beta}(1)
		$$
		which ranges over a complete system of representatives of equivalence classes of dimension vectors $\beta$ of $\hat{\Gamma}$ that are compatible with $\alpha$.
	\end{thm}
Theorem \ref{t:main} can be proved for an indivisible dimension vector $\alpha$ by applying torus localization to the moduli space $M_\lambda(\Gamma,\alpha)$ of representations of the deformed preprojective algebra $\Pi^\lambda(\Gamma)$ (see \cite[Section 3.2.2]{Weist:15}). When $\lambda$ is generic the Poincar\'{e} polynomial of $M_\lambda(\Gamma,\alpha)$ is shown to be equal to $a_{\Gamma,\alpha}(q)$ by Crawley-Boevey--Van den Bergh \cite{CBvdB:04} whence its Euler characteristic equals $a_{\Gamma,\alpha}(1)$. The fixed points under a suitable torus action can be identified with moduli $M_\lambda(\hat{\Gamma},\beta)$. The localization principle then proves the theorem in the indivisible case.

	The proof of the general case uses the varieties $\mathcal{M}_\sigma(\Gamma,\alpha)$ introduced in the proof of the Kac conjecture in \cite{HLV:13}. Their cohomology carries a natural action of the Weyl group $W$ of a maximal torus of $\GL_\alpha$. The generating series of the anti-invariant part of the cohomology equals the Kac polynomial. We find a similar torus action on $\mathcal{M}_\sigma(\Gamma,\alpha)$ which commutes with the $W$-action on cohomology. Again we describe the components of the fixed point locus (see Theorem \ref{t:conn_comps}) in terms of $\mathcal{M}_\sigma(\hat{\Gamma},\beta)$ and prove that the cohomology of the fixed point locus identifies---as a $W$-representation---with a sum of induced representations from the Weyl groups of the coverings (see Theorem \ref{t:ind_rep}). Arguing that the localization isomorphism is compatible with the Weyl group action in our setup (this is a bit of an awkward business as $W$ does not act on the quiver variety but only on its cohomology; see Proposition \ref{p:loc}), we conclude that Theorem \ref{t:main} holds for arbitrary dimension vectors.
Applying Theorem \ref{t:main} iteratively we obtain the following corollary
	\begin{cor} \label{c:main}
		The value of the Kac polynomial at one of $\Gamma$ attached to a dimension vector $\alpha$ is the sum
		$$
			a_{\Gamma,\alpha}(1) = \sum_{\beta} a_{\tilde{\Gamma},\beta}(1)
		$$
		which ranges over a complete system of representatives of equivalence classes of dimension vectors $\beta$ of the universal covering quiver $\tilde{\Gamma}$ of $\Gamma$ that are compatible with $\alpha$.
	\end{cor}

	The value $a_{\Gamma,\alpha}(1)$ is closely related to the number of indecomposable tree modules. A consequence of Theorem \ref{t:main} is Corollary \ref{kaceuler3} stating that the number of indecomposable tree modules of dimension $\alpha$ equals $a_{\Gamma,\alpha}(1)$, provided that all compatible roots $\beta$ of the universal covering quiver are exceptional. As every finite connected subquiver of the universal covering naturally defines 
 an indecomposable tree module which is exceptional as a representation of $\tilde\Gamma$, the number of such subquivers of a fixed dimension type gives a lower bound for the Kac polynomial at one. This is a consequence of the fact that the Kac polynomial is simply $1$ for exceptional roots. Indecomposable tree modules of this kind are also called cover-thin. We apply this considerations to the generalized Kronecker quiver $K(m)$. Statements about the growth behavior of the number of cover-thin tree modules of $K(m)$ then yield that $a_{K(m),n(d,e)}(1)$ grows at least exponentially in $n$ if $(d,e)$ is a root.
	
		Kinser and Derksen sketch a proof of the above theorem in an unpublished note using entirely different methods. The coprime case was also treated in the habilitation thesis \cite{Weist:15} of the second author.
	
	\subsection*{Acknowledgements}
	The authors would like to thank R.\ Kinser, S.\ Mozgovoy and M.\ Reineke for valuable discussions and remarks. While doing this research, H.F.\ was supported by the DFG SFB / Transregio 45 ``Perioden, Modulr\"aume und Arithmetik algebraischer Variet\"aten''.

	\section{Terminology} \label{s:terminology}
	
	Let $\Gamma$ be a quiver. Let $\Gamma_0$ be its set of vertices and $\Gamma_1$ be the set of arrows; both assumed to be finite. A representation $M$ of $\Gamma$ over the field $k$ consists of a tuple of finite-dimensional $k$-vector spaces $(M_i)_{i \in \Gamma_0}$ and $k$-linear maps $M_a: M_i \to M_j$ for every arrow $a: i \to j$. With the obvious notion of a morphism of representations of $\Gamma$, we obtain an abelian category. A representation $M$ over $k$ is called indecomposable if it cannot be decomposed as the direct sum of two proper subrepresentations. We call $M$ absolutely indecomposable if $M \otimes_k K$ is an indecomposable representation for every finite extension $K\mid k$.
	
	If $k = \F_q$ then the number $a_{\Gamma,\alpha}(q)$ of absolutely indecomposable representations of $\Gamma$ over $\F_q$ of dimension vector $\alpha$---the dimension vector of a representation $M$ is the tuple $(\dim_k M_i)_{i \in \Gamma_0}$---is a finite number. Kac shows in \cite[\S 1.15]{Kac:83} that $a_{\Gamma,\alpha}(q)$ is a polynomial in $q$ with integer coefficients. It is called the Kac polynomial of $(\Gamma,\alpha)$.
	
	Given $\alpha \in \Z_{\geq 0}^{\Gamma_0}$, we define the vector space
	$$
		R(\Gamma,\alpha) = \bigoplus_{a: i \to j} \Hom_k(k^{\alpha_i},k^{\alpha_j}).
	$$
	On $R(\Gamma,\alpha)$, the group $\GL_\alpha = \prod_i \GL_{\alpha_i}$ acts by change of basis. The set of $\GL_\alpha$-orbits on $R(\Gamma,\alpha)$ is in natural bijection with the set of isomorphism classes of representations of $\Gamma$ of dimension $\alpha$ over $k$.
	
	The double quiver $\bar{\Gamma}$ of $\Gamma$ is obtained as follows: the set $\bar{\Gamma}_0$ is just $\Gamma_0$ but $\bar{\Gamma}_1$ is obtained from $\Gamma_1$ by adding a new arrow $a^*: j \to i$ for every arrow $a: i \to j$ in $\Gamma$. Then
	$$
		R(\bar{\Gamma},\alpha) = \bigoplus_{(a: i \to j) \in \Gamma_1} \Big( \Hom(k^{\alpha_i},k^{\alpha_j}) \oplus \Hom(k^{\alpha_j},k^{\alpha_i}) \Big).
	$$
	An element $\phi \in R(\bar{\Gamma},\alpha)$ consists of linear maps $\phi_a: k^{\alpha_i} \to k^{\alpha_j}$ and $\phi_{a^*}: k^{\alpha_j} \to k^{\alpha_i}$. Let $\gl_\alpha$ be the Lie algebra of $\GL_\alpha$ and let $\gl_\alpha^0$ be the Lie subalgebra consisting of elements $X$ whose total trace $\tr(X) = \sum_i \tr(X_i)$ is zero. The moment map $\mu: R(\bar{\Gamma},\alpha) \to \gl_\alpha^0$ is defined by
	$$
		\mu(\phi) = \sum_{a \in \Gamma_1} [\phi_a,\phi_{a^*}].
	$$
	Let $\lambda \in \Z^{\Gamma_0}$ with $\lambda \cdot \alpha = \sum_i \lambda_i\alpha_i = 0$ which we regard as a central element of $\gl_\alpha^0$.
	Elements of the fiber $\mu^{-1}(\lambda)$ are representations of $\Pi^\lambda(\Gamma) = k \bar{\Gamma}/(\sum_{a \in \Gamma_1} [a,{a^*}] - \sum_i \lambda_ie_i)$, the so-called deformed preprojective algebra of $\Gamma$. As $\mu$ is $\GL_\alpha$-equivariant and $\lambda$ is central the fiber $\mu^{-1}(\lambda)$ carries an action of $\GL_\alpha$. The $\GL_\alpha$-orbits on $\mu^{-1}(\lambda)$ are in bijection with isomorphism classes of representations of $\Pi^\lambda(\Gamma)$ of dimension vector $\alpha$.
	
	Given a quiver $\Gamma$ we define the (infinite) quiver $\hat{\Gamma}$ by
	\begin{align*}
		\hat{\Gamma}_0 &= \Gamma_0 \times \Z^{\Gamma_1} &
		\hat{\Gamma}_1 &= \Gamma_1 \times \Z^{\Gamma_1}
	\end{align*}
	where for an arrow $a: i \to j$ in $\Gamma$ and $\chi \in \Z^{\Gamma_1}$ the arrow $(a,\chi) \in \Gamma_1$ has source $(i,\chi)$ and target $(j,\chi+e_a)$ (the element $e_a$ is the respective unit vector in $\Z^{\Gamma_1}$), i.e.\ pictorially
	$$
		(a,\chi): (i,\chi) \to (j,\chi+e_a).
	$$
	The quiver $\hat{\Gamma}$ is called the universal abelian covering quiver of $\Gamma$ (see \cite[Section 3.1]{Weist:13}). 

We also recall the notion of the universal covering quiver. We denote by $W_\Gamma$ the free group generated by $\Gamma_1$. The universal covering quiver $\tilde{\Gamma}$ of $\Gamma$ is given by the vertex set
\begin{align*}\tilde{\Gamma}_0&=\Gamma_0 \times W_\Gamma& \tilde{\Gamma}_1&=\Gamma_1\times W_\Gamma\end{align*}
where $(a,w):(i,w)\to (j,wa)$ for $a:i\to j$.

Finally, we consider iterated covering quivers and define the $k$\textsuperscript{th} universal abelian covering quiver $\hat\Gamma^k$ by
	\begin{align*}
		\hat\Gamma^k_0 &= \hat\Gamma^{k-1}_0 \times \Z^{\hat\Gamma_1^{k-1}} &
		\hat\Gamma^k_1 &= \hat\Gamma_1^{k-1} \times \Z^{\hat\Gamma_1^{k-1}}.
	\end{align*}
As explained in \cite[Section 3.4]{Weist:13} there exist natural surjective morphisms $c_k:\tilde\Gamma\to\hat\Gamma^k$ which become injective on finite subquivers of $\tilde\Gamma$ for $k\gg 0$, see \cite[Proposition 3.13]{Weist:13}.

	There is a natural morphism of quivers $c: \hat{\Gamma} \to \Gamma$ which projects along $\Z^{\Gamma_1}$. Let $\Lambda = \Z^{\Gamma_0}$ and let $\hat{\Lambda}$ be the sublattice of $\Z^{\hat{\Gamma}_0}$ of those vectors $\beta = (\beta_{i,\chi})$ with finite support. We extend the map $c$ linearly to a map $c: \hat{\Lambda} \to \Lambda$, concretely
	$$
		c(\beta)_i = \sum_\chi \beta_{i,\chi}.
	$$
	A dimension vector of $\hat{\Gamma}$ is defined to be an element of $\hat{\Lambda}$ whose entries are non-negative. We say that a dimension vector $\beta$ of $\hat{\Gamma}$ is compatible with a dimension vector $\alpha$ of $\Gamma$ if $c(\beta) = \alpha$. We define an action of the group $\Z^{\Gamma_1}$ on $\hat{\Lambda}$ by letting $\xi \in \Z^{\Gamma_1}$ act on $\beta \in \hat{\Lambda}$ by $(\xi.\beta)_{i,\chi} = \beta_{i,\chi+\xi}$. Dimension vectors which lie in the same $\Z^{\Gamma_1}$-orbit are called equivalent. The map $c$ is $\Z^{\Gamma_1}$-invariant and it is clear that up to equivalence only finitely many dimension vectors of $\hat{\Gamma}$ with connected support are compatible with a given dimension vector $\alpha$ of $\Gamma$. There is also a natural morphism $c:\tilde\Gamma\to \Gamma$ with the same properties.

	\section{The Coprime Case}
	
	If $\alpha$ is coprime then the proof of Theorem \ref{t:main} is easier. It suffices to consider moduli spaces of complex representations of the deformed preprojetive algebra. Let $k = \C$ from now on. Consider the moment map $\mu: R(\bar{\Gamma},\alpha) \to \gl_\alpha^0$. As $\alpha$ is coprime, we find $\lambda \in \Z^{\Gamma_0}$ with $\lambda \cdot \alpha = 0$ such that $\lambda \cdot \alpha' \neq 0$ for every $0 \leq \alpha' \leq \alpha$ unless $\alpha'$ equals $0$ or $\alpha$. Such a $\lambda$ is called generic for $\alpha$. Let $M_\lambda(\Gamma,\alpha) = \mu^{-1}(\lambda)/\!\!/\GL_\alpha$. By combining Formula $(2.7)$ and Corollary 2.3.2 of \cite{CBvdB:04} , we have 
	\[
		a_{\Gamma,\alpha}(q)=\sum_{i=0}^d\dim H_c^{2d+2i}(M_\lambda(\Gamma,\alpha);\C)q^i,
	\]
	where we consider singular cohomology with compact supports and where $d$ is the complex dimension of $M_\lambda(\Gamma,\alpha)$. Since $M_\lambda(\Gamma,\alpha)$ is cohomologically pure, the existence of a polynomial with integer coefficients which counts the rational points yields that the odd cohomology vanishes, see \cite[Appendix A]{CBvdB:04}. In particular, we obtain $a_{\Gamma,\alpha}(1)=\chi_c(M_\lambda(\Gamma,\alpha))$. By a well-known result, we have $\chi_c(X)=\chi_c(X^T)$ for any complex variety with a torus action, see for instance \cite[Section 2.5]{CG:97} or \cite[Appendix B]{EM:97}. Here $X^T$ denotes the fixed point set. It is straightforward to transfer the results of \cite{Weist:13} to the case of the moduli spaces $M_\lambda(\Gamma,\alpha)$. Note that representations of $\Pi^\lambda(\Gamma)$ are simple if $\lambda$ is generic. This enables us to understand the corresponding fixed point components as moduli spaces attached to the universal abelian covering of $\Gamma$. More precisely, let $T:=(\C^\times)^{\Gamma_1}$ act on $R(\bar{\Gamma},\alpha)$ by
	\[
		(t_a)_a \ast (M_a,M_{a^{\ast}})_{a \in \Gamma_1}=(t_{a}M_a,t_a^{-1}M_{a^{\ast}}).
	\]
	This descends to an action on $\mu^{-1}(\lambda)$ which commutes with the usual base change action of $\GL_\alpha$. Now the same proofs as those of \cite[Section 3]{Weist:13} apply to show the following:
	
	\begin{thm}
		The set of torus fixed points $M_\lambda(\Gamma,\alpha)^T$ is isomorphic to the disjoint union of moduli spaces
		\[
			\bigsqcup_{\beta} M_\lambda(\hat{\Gamma},\beta)
		\]
		where $\beta$ ranges over all equivalence classes of dimension vectors of $\hat{\Gamma}$ compatible with $\alpha$. 
	\end{thm}
	
	Note that $\lambda$, which we regard as an element of $\Z^{\hat{\Gamma}_0}$ by setting $\lambda_{i,\chi} = \lambda_i$, is generic for every $\beta$ that is compatible with $\alpha$. This shows that Theorem \ref{t:main} holds for $\alpha$ coprime.
	Note further that every $\beta$ for which $M_\lambda(\smash{\hat{\Gamma}},\beta)$ is non-empty must have connected support by genericity of $\lambda$. Finally, Corollary \ref{c:main} follows by Remark \ref{r:iter}.

	\section{Construction of the Moduli Space in the General Case} \label{s:constn}
	
	We recall the construction of Hausel--Letellier--Rodriguez-Villegas. Consider again $\mu: R(\hat{\Gamma},\alpha) \to \gl_\alpha^0$ over the complex numbers.
	Let $T_\alpha$ be the maximal torus of $\GL_\alpha$ of tuples of invertible diagonal matrices. Let $\t_\alpha$ be the Lie algebra of $T_\alpha$. A semi-simple element of $\gl_\alpha$ is called regular if its centralizer is a maximal torus. Therefore the centralizer of a regular element of $\t_\alpha$ is $T_\alpha$. An element $\sigma \in \t_\alpha$ is called generic if $\tr(\sigma) = 0$ and if $\tr(\sigma|_V) \neq 0$ for all non-trivial $\Gamma_0$-graded subspaces $V \sub \C^\alpha$ which are stable under $\sigma$. Let $\t_\alpha^{\gen}$ be the (non-empty) open subset of regular generic elements of $\t_\alpha$. The variety
	$$
		\mathcal{M} = \mathcal{M}(\Gamma,\alpha) = \{(\phi,hT_\alpha,\sigma) \in R(\bar{\Gamma},\alpha) \times \GL_\alpha/T_\alpha \times \t_\alpha^{\gen} \mid \mu(\phi) = h\sigma h^{-1} \}/\!\!/\GL_\alpha
	$$
	is the quotient by the $\GL_\alpha$-action defined by $g (\phi,hT_\alpha,\sigma) = (g \cdot \phi, ghT_\alpha, \sigma)$. Note that the diagonally embedded $\C^\times$ acts trivially and the induced action of $\GL_\alpha/\C^\times$ is free. The map $\pi: \mathcal{M} \to \t_\alpha^{\gen}$ arising by projecting onto the third factor is surjective.
	
	\begin{thm}[{\cite[Theorem 2.1]{HLV:13}}]
		The fibers $\mathcal{M}_\sigma$ are smooth and their cohomology vanishes in odd degrees.
	\end{thm}
	
	The Weyl group $W = W_\alpha = N_{\GL_\alpha}(T_\alpha)/T_\alpha \cong \prod_i S_{\alpha_i}$ acts on $\mathcal{M}$ via
	$$
		w.(\phi, hT_\alpha, \sigma) = (\phi,h\dot{w}^{-1}T_\alpha,\dot{w}\sigma\dot{w}^{-1})
	$$
	where $\dot{w}$ is the permutation matrix defined by $w$ (or any other representative of $w$ in $N_{\GL_\alpha}(T_\alpha)$). We will drop the dot in the notation for convenience. This gives isomorphisms $w: \mathcal{M}_\sigma \to \mathcal{M}_{w\sigma w^{-1}}$.
	
	\begin{thm}[{\cite[Theorem 2.3]{HLV:13}}] 
		For any $\sigma \in \t_\alpha^{\gen}$ the cohomology group $H_c^i(\mathcal{M}_\sigma;\C)$ becomes in a natural way a representation of $W$ which is up to isomorphism independent of $\sigma \in \t_\alpha^{\gen}$.
	\end{thm}

	In \cite{HLV:13} this result is stated for the cases that the ground field has large positive characteristic or is the complex numbers and with coefficients in $\bar{\Q}_\ell$. For our purposes, it will be sufficient to consider the complex case and $\C$-coefficients. The above theorem follows from a result of Maffei \cite[Lemma 48]{Maffei:02} which shows that $R^i\pi_!\Z$, and hence also $R^i\pi_!\C$, is constant. 

As it is useful for our purposes we will explain how the $W$-representation arises.
	
	\begin{rem} \label{r:W-action}
		Let $f: Y \to X$ be a continuous map of locally compact topological spaces. Let $W$ be a finite group which acts on both $X$ and $Y$ such that $f$ is $W$-equivariant. Let $\mathcal{F}$ be a sheaf of complex vector spaces on $Y$ with a $W$-linearization $\phi: \mathrm{act}^*\mathcal{F} \smash{\xto{}{\cong}} \pr_2^*\mathcal{F}$ (where $\mathrm{act}: W \times Y \to Y$ is the action map and $\pr_2: W \times Y \to Y$ is the projection); see \cite[\S 1.3]{GIT:94} for the definition of a linearization. We obtain isomorphisms $\phi_w: w^*\mathcal{F} \to \mathcal{F}$ with $\phi_{w_1w_2} = \phi_{w_2} \circ w_2^*\phi_{w_1}$. 
		
		Consider the higher direct images $R^if_!\mathcal{F}$ with compact supports. The definitions and results on cohomology of sheaves can, for instance, be found in \cite[Chapters II, III]{KS:13}. Suppose that $R^if_!\mathcal{F}$ is constant. As thus for every $x \in X$, the natural map $\Gamma(X;R^if_!\mathcal{F}) \to (R^if_!\mathcal{F})_x = H_c^i(Y_x;\mathcal{F})$ is an isomorphism we get an isomorphism $i_{x,x'}: H_c^i(Y_x;\mathcal{F}) \to H_c^i(Y_{x'};\mathcal{F})$ for every two points $x,x' \in X$. The linearization gives isomorphisms $\phi_w: H_c^i(Y_{wx};\mathcal{F}) = H_c^i(Y_x;w^*\mathcal{F}) \to H_c^i(Y_x;\mathcal{F})$. It is easy to verify from the cocyle conditions and the compatibility of the isomorphisms $i_{x,x'}$ that $\rho^i: W \to \GL(H_c^i(Y_x;\mathcal{F}))$ defined by
		$$
			\rho^i(w): H_c^i(Y_x;\mathcal{F}) \xto{}{\phi_{w^{-1}}} H_c^i(Y_{wx};\mathcal{F}) \xto{}{i_{wx,x}} H_c^i(Y_x;\mathcal{F}) 
		$$
		is a representation of $W$. If $\mathcal{F}$ is constructible, the assumption $R^if_*\mathcal{F}$ be constant induces a representation $W \to \GL(H^i(Y_x;\mathcal{F}))$ in the same way.
	\end{rem}
	
	The central result of \cite{HLV:13} is the description of the Kac polynomial as the generating series of the alternating part of the graded $W$-representation $H_c^*(\mathcal{M}_\sigma;\C)$. More precisely, they show:
	
	\begin{thm}[see {\cite[Theorem 1.4]{HLV:13}}]
		The Kac polynomial $a_{\Gamma,\alpha}(q)$ coincides with
		$$
			\sum_{i=0}^d \dim \left( H_c^{2i+2d}(\mathcal{M}_\sigma; \C)_{\sign} \right) q^i,
		$$
		where $d$ is the complex dimension of $\mathcal{M}_\sigma$ and the subscript ``$\sign$'' denotes the alternating component of the cohomology regarded as a $W$-representation.
	\end{thm}

	\section{Torus Action} \label{s:t_actn}
	
	Let $T = (\C^\times)^{\Gamma_1}$ act on $R(\bar{\Gamma},\alpha)$ via
	\begin{align*}
		(t.\phi)_a &= t_a\phi_a & (t.\phi)_{a^*} = t_a^{-1}\phi_{a^*}
	\end{align*}
	for any $t = (t_a)_{a \in \Gamma_1} \in T$ and $\phi \in R(\bar{\Gamma},\alpha)$. This $T$-action commutes with the action of $\GL_\alpha$ which implies that we get an action of $T$ on $\mathcal{M}$ by
	$$
		t.(\phi,hT_\alpha,\sigma) = (t.\phi,hT_\alpha,\sigma).
	$$
	The $T$-action on $\mathcal{M}$ commutes with the $W$-action whence $W$ acts on the fixed point locus $\mathcal{M}^T$.
	
	We analyze the fixed point locus and the $W$-action on it. Let $(\phi,hT_\alpha,\sigma) \in \mathcal{M}^T$. This means for all $t \in T$ there exists $g \in \GL_\alpha$ with
	$$
		(t.\phi,hT_\alpha,\sigma) = (g\cdot \phi, ghT_\alpha, \sigma)
	$$
	or, in other words, $t_a\phi_a = g_j\phi_ag_i^{-1}$ for all $a:i \to j$ and $h_i^{-1}g_ih_i \in T_{\alpha_i}$ for all $i$.
	As $\GL_\alpha/\C^\times$ acts freely on the total space of the quotient $\mathcal{M}$, the element $g$ is uniquely determined by $t$ up to a scalar. Using the arguments from \cite{Weist:13} we deduce that there exists a homomorphism $\psi: T \to \GL_\alpha$, unique up to the diagonally embedded $\C^\times$, with $t.(\phi,hT_\alpha) = \psi(t)\cdot (\phi,hT_\alpha)$. The $i$\textsuperscript{th} component $\psi_i: T \to \GL_{\alpha_i}$ of $\psi$ induces a $T$-action on $\C^{\alpha_i}$ and therefore a weight space decomposition
	$$
		\C^{\alpha_i} = \bigoplus_{\chi \in X(T)} V_{i,\chi}.
	$$
	The character group $X(T)$ of $T$ is precisely $\Z^{\Gamma_1}$. For a weight vector $v_\chi \in V_{i,\chi}$ and an arrow $a: i \to j$ we get
	\begin{align*}
		t_a\phi_a(v_\chi) &= (t.\phi)_a(v_\chi) \\ 
		&= \psi_j(t)\phi_a\psi_i(t)^{-1}(v_\chi) \\
		&= \chi(t)^{-1} \psi_j(t) \phi_a(v_\chi)
	\end{align*}
	or, in other words, $\phi_a(v_\chi) \in V_{j,\chi+e_a}$. It is shown analogously that $\phi_{a^*}(V_{j,\chi}) \sub V_{i,\chi-e_a}$. These considerations show that $\phi$ can be regarded as a representation of the double of the covering quiver $\hat{\Gamma}$ of dimension vector $\beta$ with $\beta_{i,\chi} = \dim V_{i,\chi}$.
	
	Let $\chi^1,\ldots,\chi^N$ be those characters for which there exists an $i$ such that the weight space $V_{i,\chi^k}$ is non-zero. Embed $\smash{\C^{\beta_{i,\chi^k}}}$ as the subspace of $\C^{\alpha_i}$ spanned by the unit vectors 
	$$
		e_{i,(\beta_{i,\chi^1}+\ldots+\beta_{i,\chi^{k-1}}+1)},\ldots,e_{i,(\beta_{i,\chi^1}+\ldots+\beta_{i,\chi^k})}
	$$ 
	and consider $\GL_\beta = \prod_{i,\chi} \GL_{\beta_{i,\chi}}$ as a subgroup of $\GL_\alpha$ via this direct sum decomposition. As there exists $g \in \GL_\alpha$ with $g(V_{i,\chi^k}) = \smash{\C^{\beta_{i,\chi^k}}}$ which is unique up to a (unique) element of $\GL_\beta$, we may, by passing from $(\phi,hT_\alpha,\sigma)$ to $(g \cdot\phi,ghT_\alpha,\sigma)$, assume without loss of generality that $V_{i,\chi^k} = \smash{\C^{\beta_{i,\chi^k}}}$.
	
	For a number $r \in \{1,\ldots,\alpha_i\}$, let $k$ be the unique index with $e_{i,r} \in V_{i,\chi^k}$. Write $h_i(e_{i,r})$ as $\sum_\chi w_\chi$ with $w_\chi \in V_{i,\chi}$. As $\smash{h_i^{-1}}\psi_i(t)h_i$ lies in $T_{\alpha_i}$, the vector
	$$
		(h_i^{-1}\psi_i(t)h_i)(e_{i,r}) = h_i^{-1}( \sum_\chi \chi(t) w_\chi)
	$$
	lies in the span of $e_{i,r}$. Precisely one summand $w_\chi$ is non-zero. To see this, assume otherwise. Take a character $\xi$ for which $w_\xi \neq 0$ and observe that $h_i(e_{i,r})$ would by assumption not be a multiple of $w_\xi$. Choose a one-parameter subgroup $\lambda$ of $T$ which vanishes on $\xi$ and is positive on all other $\chi$ which occur in the summation $\sum_\chi w_\chi$. Then
	$$
		(h_i^{-1}\psi_i(\lambda(z))h_i)(e_{i,r}) = h_i^{-1}(\sum_\chi z^{\langle \lambda,\chi \rangle} w_\chi) \xto{z \to 0}{} h_i^{-1}(w_\xi)
	$$
	which does not lie in the span of $e_{i,r}$. A contradiction.
	
	We deduce that for every $r$ as above there exists an $l$ such that $h_i(e_{i,r}) \in V_{i,\chi^l}$. This implies the existence of a $w \in W$ such that $h\dot{w}^{-1} \in \GL_\beta$ (and this Weyl group element $w$ is unique up to an element of $W_\beta = N_{\GL_\beta}(T_\alpha)/T_\alpha$). These considerations show that the fixed point locus $\mathcal{M}^T$ is the disjoint union
	\begin{equation} \label{e:disj_union}
		\bigsqcup_\beta \bigsqcup_{\bar{w} \in W/W_\beta} \mathcal{M}_{\beta,\bar{w}}
	\end{equation}
	over a full system of representatives of dimension vectors $\beta$ which are compatible with $\alpha$, where
	$$
		\mathcal{M}_{\beta,\bar{w}} = \{ (\phi,hT_\alpha,\sigma) \in R(\bar{\hat{\Gamma}},\beta) \times \GL_\alpha/T_\alpha \times \t_\alpha^{\gen} \mid h\dot{w}^{-1} \in \GL_\beta,\ \mu(\phi) = h\sigma h^{-1} \} /\!\!/ \GL_\beta
	$$
	and where $\dot{w} \in N_{\GL_\alpha}(T_\alpha)$ is a representative of $\bar{w} \in W/W_\beta  \cong N_{\GL_\alpha}(T_\alpha)/N_{\GL_\beta}(T_\alpha)$. An element $w' \in W$ applied to $(\phi,hT_\alpha,\sigma) \in \mathcal{M}_{\beta,\bar{w}}$ yields $(\phi,hw'^{-1}T_\alpha,w'\sigma w'^{-1})$ which lies in $\mathcal{M}_{\beta,\bar{ww'^{-1}}}$. This shows that the the disjoint union $\bigsqcup_{\bar{w} \in W/W_\beta} \mathcal{M}_{\beta,\bar{w}}$ agrees---as a $W$-variety---with the associated fiber bundle $\mathcal{M}(\smash{\hat{\Gamma}},\beta) \times^{W_\beta} W$ (cf.\ \cite[Section 3.3]{Serre:58} or \cite[Section 2.1]{Slodowy:80}).
	We have proved
	
	\begin{thm} \label{t:conn_comps}
		The fixed point locus $\mathcal{M}(\Gamma,\alpha)^T$ is, as a $W$-variety, isomorphic to the disjoint union
		$$
			\bigsqcup_\beta \mathcal{M}(\hat{\Gamma},\beta) \times^{W_\beta} W
		$$
		over all dimension vectors $\beta$ up to equivalence which are compatible with $\alpha$. This isomorphism commutes with the natural maps to $\t_\alpha^{\gen}$.
	\end{thm}
	
	Denote by $\pi_{\beta,\bar{w}}$ the natural map $\mathcal{M}_{\beta,\bar{w}} \to \t_\alpha^{\gen}$. Formula (\ref{e:disj_union}) implies that the higher direct image $R^i(\pi|_{\mathcal{M}^T})_!\C$ of the constant sheaf $\C$ on $\mathcal{M}^T$ is the direct sum
	$$
		\bigoplus_\beta \bigoplus_{\bar{w} \in W/W_\beta} R^i(\pi_{\beta,\bar{w}})_!\C
	$$ 
	and is therefore constant (using the proof of \cite[Theorem 2.3]{HLV:13}). The identification of $\bigsqcup_{\bar{w}} \mathcal{M}_{\beta,\bar{w}}$ with $\mathcal{M}(\hat{\Gamma},\beta) \times^{W_\beta} W$ yields that $R^i(\pi|_{\mathcal{M}^T})_!\C$ is isomorphic to
	$$
		\bigoplus_\beta \Ind_{W_\beta}^W R^i(\pi_\beta)_!\C
	$$
	where $\pi_\beta: \mathcal{M}(\hat{\Gamma},\beta) \to \t_\alpha^{\gen}$. We conclude
	
	\begin{thm} \label{t:ind_rep}
		The $W$-module $H_c^i(\mathcal{M}_\sigma(\Gamma,\alpha)^T; \C)$ is isomorphic to the direct sum 
		$$\bigoplus_\beta \Ind_{W_\beta}^W H_c^i(\mathcal{M}_\sigma(\hat{\Gamma},\beta); \C)$$ 
		over all dimension vectors $\beta$ up to equivalence which are compatible with $\alpha$.
	\end{thm}

	\begin{rem}\label{r:iter}When applying iterated localization, in a first step, we can replace $\hat\Gamma$ by $\hat\Gamma^k$ in Theorem \ref{t:ind_rep}. In a second step, we can replace $\hat\Gamma^k$ by $\tilde\Gamma$. Indeed, every root $\beta$ of $\hat\Gamma^k$ which is compatible with $\alpha$ defines a finite connected subquiver of $\hat\Gamma^k$. As recalled in Section \ref{s:terminology}, every finite subquiver of the universal covering quiver embeds into $\hat\Gamma^k$ for $k\gg 0$. As the maps $c_k$ defined there are also surjective, this means that we can choose $k$ in such a way that the supports of all compatible roots $\beta$ are finite subquivers of $\tilde\Gamma$. We refer also to \cite[Section 3.4]{Weist:13} for more details.
	\end{rem}
	
	\section{Localization Isomorphisms} \label{s:locn}
	
	Consider a smooth morphism $f: Y \to X$ of smooth complex varieties which is assumed to be equivariant with respect to a finite group $W$ that acts on both $X$ and $Y$. Now suppose that a complex torus $T = (\C^\times)^n$ acts on $Y$ in a way that $f$ is $T$-invariant and such that the $T$- and the $W$-action on $Y$ commute. We assume further that $f$ arises by base change from a (topologically) locally trivial fibration $Y' \to X'$ whose basis $X'$ is contractible (this holds in our concrete situation, see the proof of \cite[Theorem 2.3]{HLV:13} or \cite[Lemma 48]{Maffei:02}). In particular $f$ is also a locally trivial fibration. Consider the constant sheaf $\C$ on $Y$. Then the higher direct images $R^if_!\C$ and $R^if_*\C$ are constant by contractibility of $X'$. Remark \ref{r:W-action} ensures that $H_c^i(Y_x;\C)$ and $H^i(Y_x;\C)$ have a natural structure of a $W$-representation. Moreover every connected component of the fixed point set $Y^T$ is smooth. Suppose further that the induced map $f^T: Y^T \to X$ satisfies the same properties as $f$ (this is true in our setup, see Theorem \ref{t:conn_comps}). We show:
	
	\begin{prop} \label{p:loc}
		The classes $[H_c^*(Y_x;\C)]$ and $[H_c^*(Y_x^T;\C)]$ agree in the Grothendieck group $K_0(\C W)$ of finitely generated (complex) $W$-representations.
	\end{prop}
	
	We have to show that the localization isomorphism in \cite[Section 2.5]{CG:97} is compatible with the $W$-action. But as the $W$-action on cohomology of the fiber does not arise from a $W$-action on the fiber but from monodromy, we have to adjust the proof given in \cite{CG:97} relative to the basis $X$.

	\begin{proof}
		Let $ET$ be a contractible space on which $T$ acts freely and let $BT = ET/T$. Consider the cartesian diagrams
		\begin{center}
			\begin{tikzpicture}[description/.style={fill=white,inner sep=2pt}]
				\matrix(m)[matrix of nodes, row sep=1.5em, column sep=2em, text height=1.5ex, text depth=0.25ex]
				{
					$Y \times^T ET$ & $Y \times ET$ & $Y$ \\
					$X \times BT$ & $X \times ET$ & $X$. \\
				};
				\path[->, font=\scriptsize]
				(m-1-1) edge node[auto] {$f_T$} (m-2-1)
				(m-1-2) edge (m-2-2)
				(m-1-2) edge node[above] {$\pi$} (m-1-1)
				(m-1-2) edge node[auto] {$p$} (m-1-3)
				(m-2-2) edge node[above] {$\pi$} (m-2-1)
				(m-2-2) edge node[auto] {$p$} (m-2-3)
				(m-1-3) edge node[auto] {$f$} (m-2-3)
				;
			\end{tikzpicture}
		\end{center}
		By base change we see that $\pi^*R^if_{T,*}\C \cong p^*R^if_*\C$ and as $\pi$ is faithfully flat we obtain $R^if_{T,*}\C \cong \pr^*R^if_*\C$ where $\pr: X \times BT \to X$ is the projection. The composition
		$$
			Y \times^T ET \xto{}{f_T} X \times BT \xto{}{\pr} X
		$$
		yields a spectral sequence $E_2^{p,q} = (R^p\pr_*)(R^qf_{T,*})\C \Rightarrow R^{p+q}(\pr \circ f_T)_*\C$. By the above considerations we conclude that
		$$
			E_2^{p,q} \cong (R^p\pr_*)\pr^*(R^qf_*)\C \cong H^p(BT;\C) \otimes R^qf_*\C.
		$$
		This is a spectral sequence in the category of $W$-equivariant sheaves on $X$. Note that as $f_T$ is a locally trivial bundle, the composition $\pr \circ f_T$ is, too, and as $BT$ is contractible all $R^i(\pr \circ f_T)_*\C$ are constant. Then the stalk $H^i(Y_x \times^T BT;\C) = H_T^i(Y_x;\C)$ inherits the structure of a $W$-representation (cf.\ Remark \ref{r:W-action}). Applying the stalk functor to the spectral sequence $E$ from above yields the spectral sequence
		$$
			H_T^p(\pt) \otimes H^q(Y_x) \Rightarrow H_T^{p+q}(Y_x)
		$$
		(all cohomology groups are taken with complex coefficients). This is the spectral sequence associated with the fibration $ET \times^T Y_x \to BT$, but by this detour, we have shown that the differentials of this spectral sequence are $W$-linear. This shows that $[H_T^* \otimes H^*(Y_x)] = [H_T^*(Y_x)]$ in the Grothendieck group $K_0((H_T^*)W)$ of finitely generated $(H_T^*)W$-modules.
		
		On the other hand the inclusion $Y^T \to Y$ of the fixed point locus induces a natural morphism
		$$
			R^i(f_T \circ \pr)_*\C \to R^i((f^T \times \id_{BT}) \circ \pr)_*\C
		$$
		which, after taking fibers, is the pull-back $H_T^i(Y_x) \to H_T^i(Y_x^T)$ in equivariant cohomology. This shows that this map is also $W$-equivariant. The $H_T^*$-linear map $H_T^*(Y_x) \to H_T^*(Y_x^T)$ becomes an isomorphism after localizing finitely many non-trivial characters. This isomorphism doesn't preserve the grading but only the parity. Let $S \sub H_T^*$ be the multiplicative subset arising from the aforementioned characters. Then $[S^{-1}H_T^* \otimes_{H_T^*} H_T^*(Y_x)] = [S^{-1}H_T^* \otimes H^*(Y_x^T)]$ in the Grothendieck group $K_0((S^{-1}H_T^*)W)$ and thus
		$$
			[S^{-1}H_T^* \otimes H^*(Y_x)] = [S^{-1}H_T^* \otimes H^*(Y_x^T)]
		$$
		in the same group. But this implies that $[H^*(Y_x)] = [H^*(Y_x^T)]$ already in $K_0(\C W)$.
		
		Finally we apply Poincar\'e duality. 
		As $Y_x$ and all connected components of $Y_x^T$ are smooth varieties, the classes $[H^*(Y_x)]$ and $[H^*(Y_x^T)]$ agree with $[H_c^*(Y_x)]$ and $[H_c^*(Y_x^T)]$, respectively, in the Grothendieck group $K_0(\C)$. But Poincar\'{e} duality in this case comes from Verdier duality $R\Hom(Rf_!\C,\omega_X) \cong Rf_*R\Hom(\C,\omega_Y)$ which is $W$-equivariant (by interpreting it as an identity in the bounded derived category of $W$-equivariant sheaves on $X$ which is the same as $\mathcal{D}_W^b(X)$ the $W$-equivariant bounded derived category in the sense of Bernstein--Lunts \cite{BL:94} as $W$ is finite). Therefore, the identity $[H_c^*(Y_x)] = [H^*(Y_x)]$ holds also in the Grothendieck group $K_0(\C W)$. The same argument applies for the equality $[H_c^*(Y_x^T)] = [H^*(Y_x^T)]$.
	\end{proof}
	
	\section{Finishing the Proof of the Main Result} \label{s:pf}
	
	Applying Proposition \ref{p:loc} and the fact that $\mathcal{M}_\sigma$ has no odd cohomology, we observe that $H_c^*(\mathcal{M}_\sigma^T;\C) \cong H_c^*(\mathcal{M}_\sigma;\C)$ as ungraded $W$-representations. The $W$-representation $H_c^*(\mathcal{M}^T_\sigma;\C)$ decomposes by Theorem \ref{t:ind_rep} as the direct sum $\smash{\bigoplus_\beta \Ind_{W_\beta}^W H_c^*\mathcal{M}_\sigma(\hat{\Gamma},\beta); \C)}$. The sign-isotypical component of the induced $W$-representation $\smash{\Ind_{W_\beta}^W H_c^i(\mathcal{M}_\sigma(\hat{\Gamma},\beta);\C)}$ is thus the sign-isotypical component of $H_c^i(\mathcal{M}_\sigma(\hat{\Gamma},\beta);\C)$ (with respect to the $W_\beta$-action). We know by \cite[Theorem 1.4]{HLV:13} that
	\begin{align*}
		a_{\Gamma,\alpha}(1) &= \dim H_c^*(\mathcal{M}_\sigma(\Gamma,\alpha);\C)_{\sign} & a_{\hat{\Gamma},\beta}(1) &= \dim H_c^*(\mathcal{M}_\sigma(\hat{\Gamma},\beta);\C)_{\sign}
	\end{align*}
	and using that taking isotypical components and taking classes in the Grothendieck group commute, we have proved $a_{\Gamma,\alpha}(1) = \sum_\beta a_{\hat{\Gamma},\beta}(1)$, as asserted in Theorem \ref{t:main}. Now Corollary \ref{c:main} is an immediate consequence of Theorem \ref{t:main} when taking Remark \ref{r:iter} into account.

\section{First Consequences of the Main Theorem}\label{kactree}
The main result of this paper has plenty of interesting consequences which were also mentioned in \cite{Weist:15} for coprime dimension vectors. For instance the number of indecomposable tree modules, as defined in \cite{Ringel:98}, can be related to the Kac polynomial at one. Recall that a tree module of $\Gamma$ is already a representation of the universal covering quiver $\tilde \Gamma$ of $\Gamma$. We call an indecomposable tree module of $\Gamma$ cover-thin if its dimension vector $\alpha$ is of type one as a representation of $\tilde \Gamma$, i.e. if $\alpha_i\in\{0,1\}$ for all $i\in\tilde\Gamma_0$. Equivalently the coefficient quiver of a cover-thin tree module is a spanning tree of the support of its dimension vector (as representation of $\hat\Gamma$). We denote the number of indecomposable tree modules of dimension $\alpha$ by $t_\alpha$ and the number of cover-thin tree modules by $ct_\alpha$. Let $\sigma_i$ be the BGP-reflection at $i$ introduced in \cite{BGP:73}.
\begin{lem}\label{refltree}
Every indecomposable tree module $M$ which is cover-thin is exceptional as a representation of the universal covering quiver. In particular, we have that $\sigma_i M$ is also an indecomposable tree module for any sink (resp. source) $i\in \Gamma_0$.
\end{lem}
\begin{proof}
The first part follows because every exceptional representation $M$ is Schurian and, moreover, because $\Sc{\udim M}{\udim M}=\dim_k\Hom(M,M)-\dim_k\Ext(M,M)=1$. In particular, $\sigma_i M$ is also exceptional as a representation of the universal covering quiver and thus a tree module by the main result of \cite{Ringel:98}.
\end{proof}

\begin{cor}\label{kaceuler3}
Let $\alpha$ be a dimension vector such that all equivalence classes of compatible dimension vectors with connected support consist of exceptional roots. Then the number of indecomposable tree modules of dimension vector $\alpha$ is equal to the Kac polynomial at one.
\end{cor}
\begin{proof}
For the Kac polynomial of a real root $\tilde\alpha$, we have $a_{\tilde\Gamma,\tilde\alpha}(q)=a_{\tilde\Gamma,\tilde\alpha}(1)=1$. Thus it suffices to show that the number of indecomposable tree modules is equal to the number of compatible roots.

By the main result of \cite{Ringel:98}, every exceptional representation is an indecomposable tree module and thus a representation of the universal covering quiver, say of dimension $\tilde\alpha$. Thus every compatible dimension vector gives rise to an indecomposable tree module. Conversely every indecomposable tree module $T$ yields a root $\tilde\alpha$ of $\tilde \Gamma$ which is compatible with $\alpha$. Since $\tilde\alpha$ is exceptional by assumption, $T$ is up to isomorphism the only representation of dimension $\tilde\alpha$.
\end{proof}

Using iterated localization we also re-obtain the following result:
\begin{cor}[{\cite[Corollary 4.4]{MR:14}}]
If $\alpha$ is a dimension vector of $\Gamma$ such that $\alpha_i=1$ for all $i\in \Gamma_0$, the Kac polynomial at one is equal to the number of spanning trees of $\Gamma$.
\end{cor}
Let $K(m)$ be the generalized Kronecker quiver with two vertices $i,j$ and $m$ arrows $a_l:i\to j$. In this case, the main result together with the following proposition enables us to investigate the asymptotic behaviour of the Kac polynomial at one.
\begin{prop}[{\cite[Proposition 3.2.6]{Weist:15}}]\label{numberoftreesKronecker}
Let $n:=m-1$. The number of indecomposable cover-thin tree modules $ct_{(d,e)}$ of $K(m)$ which are of dimension $(d,e)$ is
\[\frac{1}{d}\sum_{i=1}^m\binom{m}{i}\binom{ne}{d-1}\binom{n(d-1)}{e-i}\frac{i}{e}.\]
\end{prop}

\begin{ex} If $m=3$ and $(d,e)=(d,d+1)$, it is straightforward to check that we have
\[ct_{(d,d+1)}=\frac{3}{(d+2)(d+3)}\binom{2d}{d}\binom{2(d+1)}{d+1}.\]
The respective sequence of natural numbers appears as sequence A186266 in \cite{OEIS}. It seems that there was no combinatorial interpretation of this sequence before.
\end{ex}

Applying Lemma \ref{refltree}, we thus obtain:
\begin{cor}\label{exceptional}
The number of indecomposable tree modules $t_{n(d,e)}$ of $K(m)$ grows (at least) exponentially with the dimension vector, i.e.\ for every imaginary Schur root $(d,e)$ of $K(m)$ there exists a real number $K_{(d,e)}>1$ such that $t_{n(d,e)}> K_{(d,e)}^n$.
\end{cor}

As a consequence, we obtain the following:
\begin{cor}
Let $(d,e)$ be a root of the generalized Kronecker quiver $K(m)$. Then the Kac polynomial at one grows (at least) exponentially with the dimension vector, i.e. there exists a real number $K_{(d,e)}>0$ such that $$a_{K(m),n(d,e)}(1)>K_{(d,e)}^n.$$

\end{cor}
Actually, Corollary \ref{exceptional} can also be used to show that the number of indecomposable tree modules which have an imaginary Schur root as the dimension vector grows exponentially with the dimension vector, see \cite[Theorem 3.2.8]{Weist:15}

We conclude with the following natural question, which was for instance asked in \cite[Question 7]{Kinser:13}, but also posed to the second author by W.\ Crawley-Boevey and A.\ Hubery:
\begin{que}
Do we always have $t_\alpha\geq a_{\Gamma,\alpha}(1)$?
\end{que}
The main result of the paper implies that this needs to be checked only for quivers which are trees. But actually, similar to the question of the existence of tree modules, it seems that this does not make things much easier. Many examples which can be found in the literature suggest that this is true. In the case of extended Dynkin quivers of type $\tilde D_n$, this can be checked by hand. We also conjecture that equality holds if and only if the assumptions of Corollary \ref{kaceuler3} hold.

\section{Concrete Examples} \label{s:ex}

Consider the generalized Kronecker quiver $K(3)$ and the dimension vector $(2,3)$. By use of Hua's formula \cite{Hua:00}, we obtain 
$$a_{K(3),(2,3)}=q^6+q^5+3q^4+4q^3+5q^2+3q+2$$
and thus $a_{K(3),(2,3)}(1)=19$, see \cite[Section 5]{Hua:00}. There are $18$ cover-thin tree modules of $K(3)$ of dimension $(2,3)$ which are given by 
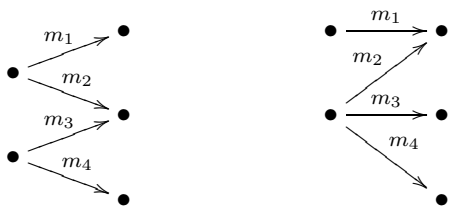
\[
\begin{xy}
\xymatrix@R5pt@C30pt{&\bullet&&\bullet\ar[r]^{m_1}&\bullet\\\bullet\ar[ru]^{m_1}\ar[rd]^{m_2}&\\&\bullet&&\bullet\ar[r]^{m_3}\ar[ruu]^{m_2}\ar[rdd]^{m_4}&\bullet\\\bullet\ar[ru]^{m_3}\ar[rd]^{m_4}&&&\\&\bullet&&&\bullet}
\end{xy}
\]
Here the arrows $m_i\in\{a_1,a_2,a_3\}$ satisfy the conditions $m_1\neq m_2\neq m_3\neq m_4$ in the first case and the conditions $m_1\neq m_2$ and $m_2,m_3,m_4$ pairwise distinct in the second case. Finally, there is one tree module which is not cover-thin and whose coefficient quiver is the one on the left hand side in the case when $m_2=m_3$ and $m_1,m_2,m_4$ are pairwise distinct. More precisely, its dimension vector is given by
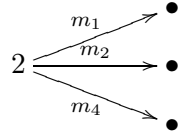
\[
\begin{xy}
\xymatrix@R10pt@C20pt{&&\bullet\\2\ar[rru]^{m_1}\ar[rr]^{m_2}\ar[rrd]_{m_4}&&\bullet\\&&\bullet}
\end{xy}
\]
i.e.\ the real root of $D_4$ of dimension $(2,1,1,1)$. Thus, the number of indecomposable tree modules is $19$.

We consider the generalized Kronecker quiver $K(4)$ and the root $(2,4)$. In this case, the number of indecomposable tree modules, which is 126 (120 cover-thin tree modules and six others), is greater than $a_{(2,4)}(1)$. Using Hua's formula \cite{Hua:00} we obtain
\[a_{K(4),(2,4)}(q)=q^{13}+q^{12}+3q^{11}+4q^{10}+8q^9+9q^8+15q^7+16q^6+20q^5+17q^4+15q^3+9q^2+5q+2\]
and thus $a_{(2,4)}(1)=125$. Up to coloring the arrows, we obtain the following subquivers and dimension vectors (where the dots indicate one-dimensional vector spaces) which give a contribution to the Kac polynomial at one:
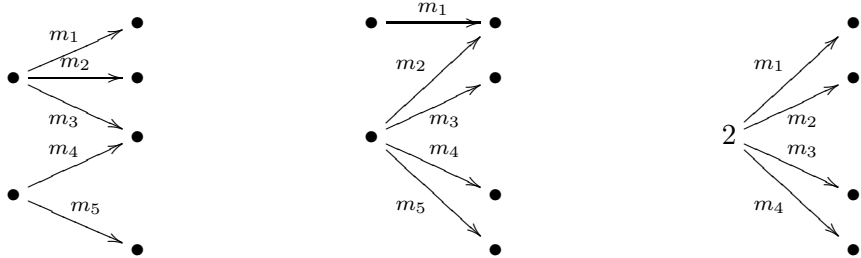
\[
\begin{xy}
\xymatrix@R10pt@C35pt{&\bullet&&\bullet\ar[r]^{m_1}&\bullet&&&\bullet\\\bullet\ar[r]^{m_2}\ar[ru]^{m_1}\ar[rd]_{m_3}&\bullet&&&\bullet&&&\bullet\\&\bullet&&\bullet\ar[ru]_{m_3}\ar[rd]^{m_4}\ar[ruu]^{m_2}\ar[rdd]_{m_5}&&&2\ar[ru]_{m_2}\ar[rd]^{m_3}\ar[ruu]^{m_1}\ar[rdd]_{m_4}\\\bullet\ar[ru]^{m_4}\ar[rd]^{m_5}&&&&\bullet&&&\bullet\\&\bullet&&&\bullet&&&\bullet}
\end{xy}
\]
It is straightforward to check that there exist (up to translation) 108 possibilities to embed the first quiver into $\widetilde{K(4)}$ (resp. to color the arrows with four different colors $m_1,m_2,m_3,m_4$ such that $m_1,m_2,m_3$ are pairwise distinct and $m_3\neq m_4\neq m_5$). For the second one we have $12$ possibilities and, for the last one, there exists only one possible embedding. Since the first two dimension vectors are real roots, the Kac polynomials are $1$. The quiver and the dimension vector considered in the last case is a quiver of type $\tilde D_4$ together with the unique imaginary Schur root $\delta$. Thus the Kac polynomial is $q+4$. This can be checked by classifying the absolutely indecomposable representations, which coincide with the indecomposables in this case, up to isomorphism. In summary, we obtain
\[a_{K(4),(2,4)}(1)=108\cdot 1+12\cdot 1+1\cdot 5=125.\]

Finally, we consider the quiver $L_g$ with only one vertex $i$ and $g$ loops. If $\alpha_i\leq 5$, all non-empty moduli spaces appearing are points and the Kac polynomials of the compatible dimension vectors are one. In particular, the number of indecomposable tree modules coincides with $a_{L_g,\alpha_i}(1)$.
The first non-trivial moduli space appears for $\alpha_i=6$. Also in this case, we need to consider the Kac polynomial of the imaginary Schur root $\delta=(2,1,1,1,1)$ of $\tilde D_4$ with Kac polynomial $a_{\tilde L_g,\delta}(q)=q+4$. Taking into account the different possibilities of coloring the arrows of $\tilde D_4$ with the colors $\{1,\ldots,g\}$ and orienting the arrows, this gives the contribution 
$$5\cdot\left(2^4{g \choose 4}+3\cdot 2^2{g\choose 3}+{g\choose 2}\right)$$ to the Kac polynomial at one (we can choose four, three or two different arrows out of the $g$ arrows of the original quiver). 
Note that, in the universal covering, the following orientations do not appear
\[\bullet\xlongrightarrow{a}\bullet\xlongleftarrow{a}\bullet\quad \bullet\xlongleftarrow{a}\bullet\xlongrightarrow{a}\bullet\]

Since there are six indecomposable tree modules of dimension $\delta$ of $\tilde D_4$, one checks that this indeed fills the gap between $a_{L_g,6}(1)$, see \cite[Section 1]{HRV:09}, and the number of indecomposable tree modules of dimension $\alpha_i=6$, see \cite[Section 4.1]{Kinser:13}.

\section{Conjecture on the Asymptotic Behavior of the Kac Polynomial at One}\label{douglas}
The following is based on a conjecture of M. Douglas concerning moduli spaces of stable representations of generalized Kronecker quivers, see \cite[Section 6.1]{Weist:13}. It generalizes \cite[Conjecture 4.1.2]{Weist:15} to arbitrary dimension vectors:
\begin{conj}\label{conjkac}
There exists a continuous function $f:\R \Gamma_0\to\R$ such that
\[f(\alpha)=\lim_{n\to\infty}\frac{\ln(a_{\Gamma,n\alpha}(1))}{n}\]
for all dimension vectors $\alpha\in\N \Gamma_0$. 
\end{conj}

If a function as predicted in Conjecture \ref{conjkac} exists, Proposition \ref{numberoftreesKronecker} immediately yields a lower bound in the case of the Kronecker quiver and for coprime  dimension vectors $(d,e)$ such that $d\leq e\leq (m-1)d+1$.

\begin{lem}
Let $(d,e)$ be a root of the Kronecker quiver such that $d\leq e\leq (m-1)d+1$ and define $k:=e/d$ and $n:=m-1$. Then we have 
\begin{eqnarray*}\lim_{d\to\infty}\frac{a_{K(m),(d,kd)}(1)}{d}&\geq&\lim_{d\to\infty}\frac{ct_{(d,kd)}}{d}\\&=&n(k+1)\ln n+k(n-1)\ln k-(nk-1)\ln(nk-1)\\&&-(n-k)\ln(n-k).\end{eqnarray*}
\end{lem}
\begin{proof}
Obviously, we only have to consider one of the $m$ summands of the formula obtained in Proposition \ref{numberoftreesKronecker}. Then the claim follows straightforwardly when applying the Stirling formula.
\end{proof}
 Note that the numbers $ct_{(d,e)}$ are not invariant under the reflection functor. Nevertheless, the reflection functor can clearly be used to obtain a lower bound for $a_{K(m),(d,e)}(1)$ for every dimension vector $(d,e)$. Furthermore, it would be interesting to know if there are tuples $(d,e)$ for which equality holds or to know more about the contribution of cover-thin tree modules to the Kac polynomial at one.

	\bibliographystyle{abbrv}
	\bibliography{Literature}
\end{document}